\documentclass[11pt]{article}

\usepackage[margin=1in]{geometry}
\usepackage{amsmath,amssymb,amsthm,mathtools}
\usepackage{booktabs}
\usepackage{enumitem}
\usepackage[hidelinks]{hyperref}

\newtheorem{theorem}{Theorem}[section]
\newtheorem{lemma}[theorem]{Lemma}
\newtheorem{corollary}[theorem]{Corollary}

\newcommand{\wt}{\operatorname{wt}}
\newcommand{\ord}{\operatorname{ord}}

\newcommand{\Neg}{\operatorname{Neg}}

\title{Equality Cases in the Tu-Deng Conjecture Bound}
\author{Thomas W. Cusick$^a$ \footnote{email: cusick@buffalo.edu~ORCID 0000-0002-0087-8855}
\vspace{.5cm}\\
$^a$\small Department of Mathematics,
\small University at Buffalo\\
\small 244 Mathematics Bldg.,  Buffalo, NY 14260\\}
\date{}

\begin{document}

\maketitle

\begin{abstract}
The Tu-Deng Conjecture, made in 2011, asserts that for $k\geq2$ and
$1\leq t<2^k-1$, the set
\[
S_{t,k}=\left\{(a,b)\in\{0,1,\ldots,2^k-2\}^2:
 a+b\equiv t\pmod{2^k-1},\quad \wt(a)+\wt(b)<k\right\}
\]
satisfies
\[
                         |S_{t,k}|\leq 2^{k-1}.
\]
Proofs of the conjecture were given independently by the author and by
Liu, Luo and Xie in arXiv in 2026.  Let $r=\wt(t)$, and let $z=k-r$ be the
number of zeros in the $k$-bit expansion of $t$.  We prove that equality
holds in the Tu-Deng Conjecture if and only if every cyclic distance
between consecutive zeros is at least $z$.  We give several other simple
necessary and sufficient conditions, including a forbidden-pattern
condition, a packing condition for the zero positions, a composition
condition for the gaps, and equivalent conditions involving the
coefficients of the polynomial $H_W(u,v)$ in the proof of the Tu-Deng
conjecture.  We also count the equality words having a prescribed number
of zeros.
\end{abstract}

\section{History of the Tu-Deng Conjecture}
The Tu-Deng Conjecture stated in the Abstract was first stated in a $2011$
paper \cite{TuDeng} and soon led to much research in an effort to prove it.
The author of the present paper and two of his students at the time soon 
wrote a paper \cite{CLS} proving some special cases of the Tu-Deng Conjecture.
A little later the present author stated a conjecture which is implied by the
Tu-Deng Conjecture (a formal proof of that was given in \cite{ SpW}) 
and this came to be known as the Cusick Conjecture. The Cusick Conjecture
was proved for almost all values of its integer parameter $t$ in \cite{DKS}. A number 
of other papers considered the Cusick Conjecture. In \cite{SpW2} it was shown
that the Cusick Conjecture is true whenever a certain condition on the cyclic
gaps defined in the next section is true. The conjecture was finally proved by
Cheng \cite{Cheng}, who made use of a random walks idea of Tarlowski \cite{Tar}.
The paper of Cheng almost immediately led to two proofs of the Tu-Deng
Conjecture. One was by the present author \cite{Proof} and the other was by $3$ 
Chinese authors Liu, Luo and Xie \cite{LLX}, who thanked ChatGPT for assisting
in the proof.

\section{Notation and the meaning of a cyclic gap}

Fix $k\geq2$, put
\[
                         M=2^k-1,
\]
and let $1\leq t<M$.  As in the Tu-Deng conjecture (see Abstract or \cite[Conjecture 1.1]{Proof}), define
\[
S_{t,k}=\left\{(a,b)\in\{0,1,\ldots,M-1\}^2:
 a+b\equiv t\pmod M,\quad \wt(a)+\wt(b)<k\right\}.
\]
Write the $k$-bit expansion of $t$ as the cyclic word
\[
 W=\varepsilon _0\varepsilon _1\cdots\varepsilon _{k-1},
 \qquad \varepsilon_i\in\{0,1\},
\]
and put
\[
 r=\sum_{i=0}^{k-1}\varepsilon_i=\wt(t),
 \qquad z=k-r.
\]
Since $1\leq t<M$, both $r$ and $z$ are positive.

Let the zero positions, in cyclic order, be
\[
q_1,q_2,\ldots,q_z\in\mathbb Z/k\mathbb Z.
\]
Define the cyclic gaps $d_1,\ldots,d_z$ by
\[
 d_\nu=(q_{\nu+1}-q_\nu)\pmod k,
 \qquad 1\leq d_\nu\leq k,
 \qquad q_{z+1}=q_1.
\]
Thus
\[
                         d_1+\cdots+d_z=k.                 \tag{1}
\]
This convention is important: $d_\nu$ is the distance (or gap) from one zero to
the next zero, so it is one more than the number of intervening ones.
Thus, if there are $a$ ones between two consecutive zeros, the
corresponding gap is $a+1$.  For example, the two gaps
of the cyclic word $0101$ are both $2$.

Set
\[
                         g(W)=\min_{1\leq\nu\leq z}d_\nu.  \tag{2}
\]

\section{The main equivalence theorem}
The simplest necessary and sufficient condition for $|S_{t,k}|=2^{k-1}$  can be stated 
very briefly in terms of the cyclic gaps as follows:

\begin{theorem}[Cyclic gap characterization]
Let $1\leq t<2^k-1$, and let $z=k-\wt(t)$.  In the cyclic $k$-bit
expansion of $t$, let $d_1,\ldots,d_z$ be the distances between
successive zeros, where the distance includes the final zero.
Then
\[
                         |S_{t,k}|=2^{k-1}
 \quad\Longleftrightarrow\quad
                         \min_\nu d_\nu\geq z.               
\]
Equivalently, at least $z-1$ ones occur between every two consecutive
zeros.
\end{theorem}

The next theorem will give a long list of necessary and sufficient conditions.  Some are purely
combinatorial, while the last three use the polynomial notation of the proof in
\cite[Sections 2--5 and 8]{Proof}.

With the matrices from \cite[Sections 2--3]{Proof},
\[
 C_0=\begin{pmatrix}1&0\\ u&v\end{pmatrix},\qquad
 C_1=\begin{pmatrix}u&v\\ 0&1\end{pmatrix},\qquad
 M_W=C_{\varepsilon_0}C_{\varepsilon_1}\cdots C_{\varepsilon_{k-1}},
\]
put
\[
 Q_W(u,v)=\operatorname{tr}M_W-u^rv^z.
\]
Let
\[
 H_W(u,v)=\sum_{i,j\geq0}h_{i,j}(W)u^iv^j
\]
be the polynomial of \cite[equation (10)]{Proof}, defined by
\[
 Q_W(u,v)=1+(u+v-1)H_W(u,v),                                \tag{3}
\]
and define the positive deletion polynomial
\[
 R_W(u,v)=H_W+D_uH_W,
 \qquad
 D_u=(u-1)\frac{\partial}{\partial u}
                  +v\frac{\partial}{\partial v}.            \tag{4}
\]
Also define the diagonal defects
\[
 \Delta_m^-(W)=2h_{m-1,m}(W)-h_{m,m}(W),
 \qquad m\geq1.                                             \tag{5}
\]

\begin{theorem}[Equivalent equality conditions]\label{thm:main}
The following conditions are equivalent.

\begin{enumerate}[label=\textup{(\roman*)},ref=\textup{(\roman*)},leftmargin=3em]
\item \label{it:equality}
      $|S_{t,k}|=2^{k-1}$.

\item \label{it:gaps}
      Every cyclic gap satisfies
      \[
                              d_\nu\geq z.
      \]

\item \label{it:runs}
      Every zero is isolated and every maximal cyclic run of ones has
      length at least $z-1$.

\item \label{it:patterns}
      The cyclic word $W$ contains no cyclic factor
      \[
                              0\,1^j\,0
                              \qquad(0\leq j\leq z-2).
      \]
      When $z=1$, the displayed range of $j$ is empty.

\item \label{it:packing}
      If
      \[
            Z=\{i\in\mathbb Z/k\mathbb Z:\varepsilon_i=0\},
      \]
      then
      \[
             Z\cap(Z+s)=\varnothing
             \qquad(1\leq s\leq z-1).                       \tag{6}
      \]
      Equivalently, the $z$ cyclic intervals
      \[
             q+[0,z-1]\pmod k,\qquad q\in Z,
      \]
      are pairwise disjoint.

\item \label{it:composition}
      There are nonnegative integers $e_1,\ldots,e_z$ such that
      \[
             d_\nu=z+e_\nu\quad(1\leq\nu\leq z),
             \qquad
             e_1+\cdots+e_z=k-z^2.                           \tag{7}
      \]

\item \label{it:divisibility}
      The positive deletion polynomial satisfies
      \[
                              u^{z-1}\mid R_W(u,v).           \tag{8}
      \]

\item \label{it:negative}
      The Laurent polynomial $R_W(x,x^{-1})$ has no negative
      powers of $x$.

\item \label{it:defects}
      Every diagonal defect vanishes:
      \[
            \Delta_m^-(W)=0\quad(m\geq1),
      \]
 or, equivalently,
      \[
            2h_{m-1,m}(W)=h_{m,m}(W)\quad(m\geq1).           \tag{9}
      \]
\end{enumerate}

If $z\geq2$, these conditions are also equivalent to the following
single defect test:
\begin{equation}
 r\geq z
 \quad\hbox{and}\quad
 2h_{z-2,z-1}(W)=h_{z-1,z-1}(W).                             \tag{10}
\end{equation}
For $z=1$, equality always holds.
\end{theorem}

Thus the proposed condition is exactly condition
\ref{it:gaps}.  In words, if $z=k-\wt(t)$, equality holds precisely
when at least $z-1$ ones occur between every two consecutive zeros.

\section{Polynomial facts used in the proof}

For a nonzero polynomial $P(u,v)$, let $\ord_uP$ denote the least
exponent of $u$ occurring in $P$, and put $\ord_u0=\infty$.
We record the coefficient and support facts needed below.  The point of
the lemma is that a short cyclic gap produces a strictly positive
coefficient on the first diagonal below the main diagonal.

\begin{lemma}[Gap-support lemma]\label{lem:support}
Let $g=g(W)$.

\begin{enumerate}[label=\textup{(\alph*)}]
\item
The $v$-degree of $R_W$ is at most $z-1$.

\item
If $z\geq2$, then
\[
                         \ord_u R_W=g-1.                     \tag{11}
\]
In particular,
\[
                  u^{z-1}\mid R_W
                  \quad\Longleftrightarrow\quad g\geq z.   \tag{12}
\]

\item
Suppose $z\geq2$, $r\geq z$, and $g<z$.  Then
\[
                    [u^{z-2}v^{z-1}]R_W>0.                  \tag{13}
\]
Equivalently,
\[
                         \Delta_{z-1}^-(W)>0.                \tag{14}
\]
\end{enumerate}
\end{lemma}

\begin{proof}
The bit-complement symmetry proved in \cite[proof of Theorem 5.2]{Proof}
interchanges $u$ and $v$.  Applying this symmetry to the top-boundary
formula in \cite[Theorem 8.3, equation (43)]{Proof} gives
\[
                         \deg_v H_W\leq z-1.
\]
None of the operations in (4) increases the $v$-degree.  Hence
\[
                         \deg_v R_W\leq z-1,                 \tag{15}
\]
which proves part (a).

We next explain the $u$-order.  The positive differential decomposition in
\cite[Theorem 5.1, equation (27)]{Proof} is
\[
 R_W=\sum_{\varepsilon_i=1}\gamma_{V(i)}
          +\sum_{\varepsilon_i=0}\delta_{V(i)},              \tag{16}
\]
where $V(i)$, $\gamma_V$, and $\delta_V$ are those of
\cite[Sections 2 and 5]{Proof}.  All the summands have nonnegative
coefficients.  The relevant part
of the six-state recurrence from \cite[Theorem 2.2, equations (6)--(7)]{Proof}
is
\[
\begin{array}{c|ccc}
 &\gamma'&\delta'&d'\\ \hline
 \text{append }1&u\gamma&\delta+v\gamma&ud\\
 \text{append }0&\gamma+d+u\delta&v\delta&vd.
\end{array}                                                   \tag{17}
\]
Put
\[
 A=\ord_u\gamma,\qquad B=\ord_u\delta,
 \qquad C=\ord_ud,
\]
with $\ord_u0=\infty$.  Taking $u$-orders in (17) gives
\[
\begin{array}{c|ccc}
 &A'&B'&C'\\ \hline
 \text{append }1&A+1&\min\{B,A\}&C+1\\
 \text{append }0&\min\{A,C,B+1\}&B&C.
\end{array}                                                   \tag{18}
\]

First suppose that the zeros of $W$ are isolated.  After a cyclic
rotation write
\[
              W=0\,1^{a_1}0\,1^{a_2}\cdots0\,1^{a_z},
              \qquad a_\nu\geq1.                             \tag{19}
\]
Applying (18) successively to the blocks in (19) gives the following
two elementary facts:
\[
\begin{aligned}
 \ord_u\delta_{V(i)}&=a_\nu,
   &&\text{if the deleted position $i$ is the zero preceding
             the run $1^{a_\nu}$},\\
 \ord_u\gamma_{V(i)}&=a_\nu,
   &&\text{if the deleted position $i$ lies in the run
             $1^{a_\nu}$}.
\end{aligned}                                                 \tag{20}
\]
Here is the short induction behind (20).  Begin with
$(A,B,C)=(\infty,\infty,0)$.  If the deleted position is the zero
preceding $1^{a_\nu}$, then $V(i)$ begins with $a_\nu$ ones.  After
these ones and the next zero, (18) gives $A=a_\nu$ and $B=\infty$;
at the first following one, $B$ becomes $a_\nu$.  Thereafter, before
every later zero one has $A,C\geq a_\nu+1$, so that the zero resets
$A$ to $a_\nu+1$ while leaving $B$ unchanged.  Thus a later one can
never lower $B$, and the final value is $B=a_\nu$.

If instead the deleted position is a one in $1^{a_\nu}$, write
$p$ for the number of ones of that run occurring at the beginning of
$V(i)$ and $q$ for the number occurring at the end.  Then
$p+q=a_\nu-1$.  After the initial $p$ ones and the next zero, (18)
gives $A=p$, $B=\infty$, and $C=p$.  At the first subsequent one,
$B$ becomes $p$ and remains $p$.  At each later zero, after the first,
the value of $A$ is reset to $p+1$: immediately before such a zero one
has $A\ge p+1$ and $C\ge p+1$, while $B+1=p+1$.  In particular, at
the zero preceding the terminal block of $q$ ones we have $A=p+1$.
The final $q$ ones increase $A$ to
$p+1+q=a_\nu$.  This proves both assertions in (20).

Since (16) is a sum of polynomials with nonnegative coefficients, (20)
implies
\[
                    \ord_uR_W=\min_\nu a_\nu=g-1.             \tag{21}
\]
If two zeros are adjacent, then $g=1$.  The calculation below shows
that $[u^0v^{z-1}]R_W>0$, and therefore $\ord_uR_W=0=g-1$ in
that case also.  This proves part (b).

It remains to prove the asserted strict coefficient.  Write the general
cyclic run decomposition as
\[
              W=1^{a_1}0^{b_1}\cdots1^{a_p}0^{b_p},
              \qquad a_\nu,b_\nu\geq1,                       \tag{22}
\]
where
\[
              a_1+\cdots+a_p=r,\qquad
              b_1+\cdots+b_p=z.
\]
Define
\[
 P(y)=\prod_{\nu=1}^{p}(1+y+\cdots+y^{a_\nu})
      =\sum_{i=0}^{r}A_i y^i.                                \tag{23}
\]
Using again the bit-complement symmetry from \cite[proof of Theorem 5.2]{Proof}
and then the top-boundary formula \cite[Theorem 8.3, equation (43)]{Proof}
gives
\[
 h_{i,z-1}(W)=A_i\quad(0\leq i<r),
 \qquad h_{r,z-1}(W)=0.                                     \tag{24}
\]
From (4),
\[
 [u^iv^j]R_W=(i+j+1)h_{i,j}-(i+1)h_{i+1,j}.                 \tag{25}
\]
Consequently, for $0\leq i<r$,
\[
 [u^iv^{z-1}]R_W=(i+z)A_i-(i+1)A_{i+1}.                    \tag{26}
\]

Put $F_a(y)=1+y+\cdots+y^a$.  The elementary identity
\[
                 F_a-(1-y)F_a'=(a+1)y^a                    \tag{27}
\]
implies
\[
 pP-(1-y)P'
 =\sum_{\nu=1}^{p}(a_\nu+1)y^{a_\nu}
       \prod_{\substack{\mu=1\\\mu\ne\nu}}^{p}F_{a_\mu}(y).
                                                                    \tag{28}
\]
Therefore
\[
\begin{split}
 zP-(1-y)P'
  ={}&(z-p)P\\
   &+\sum_{\nu=1}^{p}(a_\nu+1)y^{a_\nu}
       \prod_{\substack{\mu=1\\\mu\ne\nu}}^{p}F_{a_\mu}(y).
\end{split}                                                       \tag{29}
\]
The coefficient of $y^i$ on the left side of (29) is exactly the
right side of (26).

If some two zeros are adjacent, then $p<z$.  Since every coefficient of
$P$ from degree $0$ through degree $r$ is positive, the term $(z-p)P$
shows that
\[
                    [u^iv^{z-1}]R_W>0
                    \qquad(0\leq i<r).                       \tag{30}
\]
In particular, taking $i=0$ proves the adjacent-zero assertion used in
the proof of part (b).

Suppose instead that all zeros are isolated.  Then $p=z$, and the
one-run lengths in (22) are precisely the integers $d_\nu-1$.  If
$g<z$, some $a_\nu=g-1\leq z-2$.  Take $i=z-2$ in (29).  Since
$r\geq z$, we have $i<r$.  Moreover
$0\leq i-a_\nu\leq r-a_\nu$, and a product of polynomials
$1+y+\cdots+y^{a_\mu}$ has a positive coefficient in every degree
from $0$ through the sum of its degrees.  Hence the coefficient of
\[
 y^{i-a_\nu}
 \quad\text{in}\quad
 \prod_{\mu\ne\nu}F_{a_\mu}(y)
\]
is positive.  Thus (26) again gives
\[
                    [u^{z-2}v^{z-1}]R_W>0.
\]
Finally, putting $m=z-1$ in (25) gives
\[
 [u^{z-2}v^{z-1}]R_W
    =(z-1)\bigl(2h_{z-2,z-1}-h_{z-1,z-1}\bigr)
    =(z-1)\Delta_{z-1}^-,                                   \tag{31}
\]
which proves parts (c) and (14).
\end{proof}

\section{Proof of the equivalence theorem}

The equivalence among the elementary combinatorial conditions is
immediate, but we give the details.  Since $d_\nu-1$ is the number of
ones between consecutive zeros,
\[
 d_\nu\geq z
 \quad\Longleftrightarrow\quad
 d_\nu-1\geq z-1.                                          \tag{32}
\]
This proves the equivalence of conditions
\ref{it:gaps}, \ref{it:runs}, and \ref{it:patterns}.
Condition \ref{it:packing} says exactly that no two zero positions are
at cyclic distance $1,2,\ldots,z-1$, so it is equivalent to
condition \ref{it:gaps}.  Finally, using (1), condition
\ref{it:gaps} is equivalent to writing
\[
 d_\nu=z+e_\nu,\qquad e_\nu\geq0,
\]
where
\[
 \sum_{\nu=1}^{z}e_\nu
     =\sum_{\nu=1}^{z}d_\nu-z^2=k-z^2.
\]
Thus conditions \ref{it:gaps}--\ref{it:composition} are equivalent.

By Lemma \ref{lem:support}(b), condition \ref{it:gaps} is equivalent
to condition \ref{it:divisibility} when $z\geq2$.  When $z=1$, both
conditions are automatic, so the equivalence holds in that case as
well.

We next connect these statements to equality in the Tu-Deng bound.
The coefficient calculation in \cite[proof of Theorem 5.2]{Proof} gives
\[
 [u^{m-1}v^m]R_W=m\Delta_m^-(W).                            \tag{33}
\]
As in \cite[equation (39)]{Proof}, let $\sigma^+(W)$ and
$\sigma^-(W)$ denote the total coefficients of the nonnegative and
negative powers, respectively, of $Q_W(x/2,x^{-1}/2)$.  The exact
median identity in \cite[Lemma 8.2, equation (41)]{Proof} is
\[
 \sigma^+(W)=\frac12+
       \sum_{m\geq1}2^{-2m-1}\Delta_m^-(W),                 \tag{34}
\]
and every $\Delta_m^-(W)$ is nonnegative.  Hence
\[
 \sigma^+(W)=\frac12
 \quad\Longleftrightarrow\quad
 \Delta_m^-(W)=0\text{ for every }m\geq1.                  \tag{35}
\]

Suppose first that the gap condition holds.  From (1),
\[
                         k=\sum_{\nu=1}^{z}d_\nu\geq z^2.   \tag{36}
\]
If $z\geq2$, this gives
\[
                         r=k-z\geq z(z-1)\geq z.             \tag{37}
\]
The case $z=1$ also has $r\geq z$.  Lemma \ref{lem:support}
gives
\[
                  u^{z-1}\mid R_W,
                  \qquad \deg_vR_W\leq z-1.                \tag{38}
\]
Consequently $R_W$ has no monomial $u^{m-1}v^m$: if
$m\leq z-1$, its $u$-degree is too small, while if $m\geq z$,
its $v$-degree is too large.  Equations (33)--(35) therefore give
\[
                         \sigma^+(W)=\sigma^-(W)=\frac12.   \tag{39}
\]

Theorem 6.1, the final assertion of Theorem 6.2 (following equation (35)),
and equation (38) of \cite{Proof} combine to give
\[
 \frac{|S_{t,k}|}{2^k}
 =\Neg\left(
       Q_W(x/2,x^{-1}/2)+2^{-k}x^{2r-k}
           \right),                                        \tag{40}
\]
where $\Neg$ denotes the total coefficient of the negative powers of
$x$.  Since
\[
                         2r-k=r-z\geq0,                     \tag{41}
\]
the last monomial in (40) contributes no negative mass.  Equations
(39) and (40) give
\[
                         |S_{t,k}|=2^{k-1}.                  \tag{42}
\]

Conversely, suppose that the gap condition fails.  If $r<z$, then
$2r<k$.  Equation (44), following Theorem 8.3 in \cite{Proof}, gives
\[
                         \Delta_r^-(W)\geq2.                 \tag{43}
\]
It follows from (34) that the negative mass of
$Q_W(x/2,x^{-1}/2)$ is at most
\[
                         \frac12-2^{-2r}.
\]
The determinant monomial in (40) now has negative exponent and adds
$2^{-k}$.  Since $k>2r$,
\[
                         2^{-k}<2^{-2r},
\]
so the total negative mass in (40) is strictly smaller than $1/2$.
Therefore
\[
                         |S_{t,k}|<2^{k-1}.                  \tag{44}
\]

It remains to consider a failed gap condition with $r\geq z$.  By
Lemma \ref{lem:support}(c),
\[
                         \Delta_{z-1}^-(W)>0.
\]
Thus (34) is strict, so $\sigma^+(W)>1/2$ and
$\sigma^-(W)<1/2$.  The determinant monomial has exponent
$r-z\geq0$ and adds no negative mass.  Equation (40) again gives the
strict inequality (44).  This proves
\[
 |S_{t,k}|=2^{k-1}
 \quad\Longleftrightarrow\quad
 d_\nu\geq z\text{ for every }\nu.                          \tag{45}
\]

The same argument proves the remaining polynomial equivalences.  If
the gap condition holds, (38) shows that $R_W(x,x^{-1})$ has no
negative power.  If the gap condition fails and $r<z$, (43) and (33)
produce the term $u^{r-1}v^r$; if $r\geq z$, Lemma
\ref{lem:support}(c) produces the term $u^{z-2}v^{z-1}$.
Thus a failed gap condition always produces a negative power of $x$.
This proves the equivalence with condition \ref{it:negative}.
Equation (33), together with the already established equivalence between
the gap condition and equality and with (34)--(35), proves the equivalence
with condition \ref{it:defects}.

Finally, suppose $z\geq2$.  If equality holds, we have $r\geq z$ and
all defects vanish, so (10) holds.  Conversely, suppose (10) holds.
If the gap condition failed, Lemma \ref{lem:support}(c) would give
$\Delta_{z-1}^->0$, contrary to (10).  Thus (10) is equivalent to all
of the other conditions.  The proof of Theorem \ref{thm:main} is
complete.

\section{Consequences and examples}

\begin{corollary}[A simple numerical obstruction]\label{cor:numerical}
If $|S_{t,k}|=2^{k-1}$ and the $k$-bit word for $t$ has $z$ zeros,
then
\[
                              z^2\leq k.                     \tag{46}
\]
Equivalently, an equality word can have at most $\lfloor\sqrt{k}\rfloor$
zeros.
\end{corollary}

\begin{proof}
This is (36).
\end{proof}

If $k=z^2$, all the nonnegative integers $e_\nu$ in (7) must be zero.
Thus
\[
                          d_1=\cdots=d_z=z,                  \tag{47}
\]
and the equality words are precisely the rotations of
\[
                          (0\,1^{z-1})^z.                    \tag{48}
\]

There is also a simple enumeration formula.

\begin{corollary}[Number of equality words with $z$ zeros]
Let $N(k,z)$ be the number of $k$-bit words on the labeled cyclic
positions $\mathbb Z/k\mathbb Z$ with exactly $z$ zeros for which $|S_{t,k}|=2^{k-1}$.  Then
\[
N(k,z)=
\begin{cases}
\displaystyle
 \frac{k}{k-z(z-1)}
 \binom{k-z(z-1)}{z},&k\geq z^2,\\[2ex]
0,&k<z^2.
\end{cases}                                                   \tag{49}
\]
\end{corollary}

\begin{proof}
Assume $k\geq z^2$ and put $L=k-z^2$.  By condition
\ref{it:composition}, after distinguishing one zero, an equality word
is specified by a weak composition
\[
                         L=e_1+\cdots+e_z.
\]
There are
\[
                         \binom{L+z-1}{z-1}
\]
such compositions and $k$ choices for the position of the distinguished
zero.  Every word is counted once for each of its $z$ zeros.  Hence
\[
 N(k,z)=\frac{k}{z}\binom{k-z^2+z-1}{z-1}.
\]
Using
\[
 \frac1z\binom{n-1}{z-1}=\frac1n\binom nz,
 \qquad n=k-z(z-1),
\]
gives (49).
\end{proof}

\begin{center}
\begin{tabular}{cccc}
\toprule
$k$ & cyclic word $W$ & cyclic gaps & conclusion\\
\midrule
$4$ & $0101$      & $(2,2)$     & $|S_{t,4}|=8$\\
$5$ & $01011$     & $(2,3)$     & $|S_{t,5}|=16$\\
$5$ & $00111$     & $(1,4)$     & $|S_{t,5}|=12<16$\\
$9$ & $011011011$ & $(3,3,3)$   & $|S_{t,9}|=256$\\
\bottomrule
\end{tabular}
\end{center}

In the first example, the number of intervening ones in each gap is
$1$, but the cyclic gap is $2$.  This illustrates why the
precise convention for the word ``gap'' matters.

\end{document}